\documentclass[12pt]{amsart}
\usepackage[margin=1in]{geometry}
\usepackage{amsfonts}
\usepackage{amsmath}
\usepackage{amssymb}
\usepackage{amsthm}
\usepackage{mathrsfs}
\usepackage{enumerate}
\makeatletter
\def\imod#1{\allowbreak\mkern10mu({\operator@font mod}\,\,#1)}
\makeatother

\newtheorem{theorem}{Theorem}[section]
\newtheorem{lemma}{Lemma}[section]
\newtheorem{corollary}{Corollary}[section]
\newtheorem{fact}{Fact}[section]

\theoremstyle{definition}
\newtheorem{definition}{Definition}[section]
\newtheorem{remark}{Remark}[section]

\begin{document}

\begin{center}
\title{Unitary Cayley Graphs of Dedekind Domain Quotients}
\author{Colin Defant}
\address{University of Florida \\ 1400 Stadium Rd. \\ Gainesville, FL 32611 United States}
\email{cdefant@ufl.edu}
\end{center}
\vskip .2 in

\begin{abstract}
If $X$ is a commutative ring with unity, then the unitary Cayley graph of $X$, denoted $G_X$, is defined to be the graph whose vertex set is $X$ and whose edge set is $\{\{a,b\}\colon a-b\in X^\times\}$. When $R$ is a Dedekind domain and $I$ is an ideal of $R$ such that $R/I$ is finite and nontrivial, we refer to $G_{R/I}$ as a \emph{generalized totient graph}. We study generalized totient graphs as generalizations of the graphs $G_{\mathbb{Z}/(n)}$, which have appeared recently in the literature, sometimes under the name \emph{Euler totient Cayley graphs}. We begin by generalizing to Dedekind domains the arithmetic functions known as Schemmel totient functions, and we use one of these generalizations to provide a simple formula, for any positive integer $m$, for the number of cliques of order $m$ in a generalized totient graph. In particular, we prove that the number of cliques of order $m$ in $G_{\mathbb Z/(n)}$ is \[\prod_{k=1}^m\frac{S_{k-1}(n)}{k},\] where $S_r$ is the $r^{\text{th}}$ Schemmel totient function. 

We then proceed to determine many properties of generalized totient graphs such as their clique numbers, chromatic numbers, chromatic indices, clique domination numbers, and (in many, but not all cases) girths. We also determine the diameter of each component of a generalized totient graph. We correct one erroneous claim about the clique domination numbers of Euler totient Cayley graphs that has appeared in the literature and provide a counterexample to a second claim about the strong domination numbers of these graphs.   
\end{abstract}

\maketitle


 
\section{Introduction}
\label{intro}
Throughout this paper, we will let $R$ be an arbitrary Dedekind domain. For nonzero ideals $I$ and $J$ of $R$, we will make repeated use of the fact that $\vert R/IJ\vert=\vert R/I\vert\cdot\vert R/J\vert$ regardless of whether or not $I$ and $J$ are relatively prime ideals. Furthermore, if $I$ factors into powers of prime ideals as $I=P_1^{\alpha_1}\cdots P_t^{\alpha_t}$, then $a+I\in (R/I)^\times$ if and only if $a\not\in P_i$ for each $i\in\{1,\ldots,t\}$. When $R/I$ is finite, we will refer to $\vert R/I\vert$ as the index of $I$ in $R$. If $I\neq (0)$ and $R/I$ is finite and nontrivial, then we will let $Q(I)$ denote the minimum of $\vert R/P\vert$ as $P$ ranges over all prime ideal divisors of $I$.

The well-known Euler totient function $\phi\colon\mathbb{N}\rightarrow\mathbb{N}$ maps a positive integer $n$ to the number of positive integers that are less than or equal to $n$ and relatively prime to $n$. In other words, $\phi(n)=\vert (\mathbb{Z}/(n))^\times\vert$. In 1869, V. Schemmel introduced a class of functions $S_r$, now known as Schemmel totient functions, which generalize the Euler totient function. For all positive integers $r$ and $n$, $S_r(n)$ counts the number of positive integers $k\leq n$ such that $\gcd(k+i,n)=1$ for all $i\in\{0,1,\ldots,r-1\}$. Thus, $S_1=\phi$. We will convene to let $S_0(n)=n$ for all $n\geq 1$. For each $r\geq 0$, $S_r$ is a multiplicative arithmetic function that satisfies 
\begin{equation} \label{Eq1} 
S_r(p^{\alpha})=\begin{cases} 0, & \mbox{if } p\leq r; \\ p^{\alpha-1}(p-r), & \mbox{if } p>r \end{cases} 
\end{equation} 
for all primes $p$ and positive integers $\alpha$ \cite{Schemmel69}. We may consider an extension of the Euler totient function to Dedekind domains, defining $\varphi(R/I)=\vert (R/I)^\times\vert$ whenever $I$ is an ideal of $R$ such that $R/I$ is finite (we will use the symbol $\phi$ to represent the traditional Euler totient function whose domain is $\mathbb Z^+$, and we will use $\varphi$ to represent this mapping from quotients of Dedekind domains to $\mathbb{Z}$). Thus, $\varphi(\mathbb{Z}/(n))=\phi(n)$ for all $n\in\mathbb{N}$. Later, we will define two classes of functions that will each serve to extend the Schemmel totient functions to Dedekind domains. 

Our primary goal is to study properties of certain unitary Cayley graphs. If $X$ is a commutative ring with unity, then the unitary Cayley graph of $X$, denoted $G_X$, is defined to be the graph whose vertex set is $X$ and whose edge set is $\{\{a,b\}\colon a-b\in X^\times\}$. The unitary Cayley graphs $G_{\mathbb{Z}/(n)}$ for $n\in\mathbb{Z}^+$ have been named ``Euler totient Cayley graphs" \cite{Madhavi10,Maheswari13a,Maheswari13b}. Several researchers have shown that the graph $G_{\mathbb{Z}/(n)}$ contains exactly $\displaystyle{\frac{1}{6}}n\phi(n)S_2(n)$ triangles \cite{Dejter95,Klotz07,Madhavi10}, and Manjuri and Maheswari have studied Euler totient Cayley graphs in the context of domination parameters \cite{Maheswari13a,Maheswari13b}.  Klotz and Sander have studied, among other properties, the diameters and eigenvalues of Euler totient Cayley graphs \cite{Klotz07}, and their paper gives a list of references to other results related to these graphs. 
\par 
We define a \emph{generalized totient graph} to be a unitary Cayley graph $G_{R/I}$, where $I$ is an ideal of the Dedekind domain $R$ and $R/I$ is finite and nontrivial. We seek to gain information about many of the properties of generalized totient graphs. In particular, we will use one of our two extensions of the Schemmel totient functions to give a formula, for each positive integer $m$, for the number of cliques of order $m$ in a given generalized totient graph. This formula, which apparently has not yet appeared anywhere in the literature even for Euler totient Cayley graphs, will allow us to determine the clique domination numbers of generalized totient graphs and correct an erroneous claim that Manjuri and Maheswari have made regarding this topic. We will build upon the work of Klotz and Sander, who have determined the diameters of Euler totient Cayley graphs. We end the paper with suggestions for further research and a counterexample to a claim that Manjuri and Maheswari have made regarding the strong domination numbers of Euler totient Cayley graphs.
\section{Extending the Schemmel Totient Functions}  
Our first extension of the Schemmel totient functions is inspired by our original definition of $S_r(n)$, for any given $r,n\in\mathbb{N}$, as the number of positive integers $k\leq n$ such that $\gcd(k+i,n)=1$ for all $i\in\{0,1,\ldots,r-1\}$. Implicit in the following definition is the fact that every Dedekind domain has a unity element, which we will denote $1$. Furthermore, a positive integer $k$, when used to denote an element of a Dedekind domain, will be understood to represent the sum $\underbrace{1+1+\cdots+1}_{k\  \text{times}}$.    
\begin{definition} \label{Def2.1} 
Let $I$ be an ideal of $R$ such that $R/I$ is finite. For any positive integer $r$, we define the set $L_r(R/I)$ by 
\[L_r(R/I)=\{a+I\in R/I\colon a+i+I\in (R/I)^\times\hspace{1mm}\forall\hspace{1mm} i\in\{0,1,\ldots,r-1\}\}.\] 
Furthermore, we define $\mathscr{S}_r(R/I)$ by $\mathscr{S}_r(R/I)=\vert L_r(R/I)\vert$.   
\end{definition}  
\begin{remark} \label{Rem2.1}  
Setting $R=\mathbb{Z}$ and $I=(n)$ in Definition \ref{Def2.1} yields $\mathscr{S}_r(\mathbb{Z}/(n))$ $=S_r(n)$. Observe that $\mathscr{S}_r(R/R)=1$. Also, note that $L_1(R/I)=(R/I)^\times$, so $\mathscr{S}_1(R/I)=\vert (R/I)^\times\vert=\varphi(R/I)$.   
\end{remark} 
The following two theorems show that the functions $\mathscr{S}_r$ can be evaluated using a formula similar to \eqref{Eq1}. 
\begin{theorem} \label{Thm2.1} 
Let $I$ and $J$ be relatively prime nonzero  ideals of $R$ such that $R/I$ and $R/J$ are finite. Then $\mathscr{S}_r(R/IJ)=\mathscr{S}_r(R/I)\mathscr{S}_r(R/J)$ for all positive integers $r$.  
\end{theorem} 
\begin{proof} 
Fix some positive integer $r$. Consider the natural ring homomorphisms $\psi_1\colon R/IJ\rightarrow R/I$ and
$\psi_2\colon R/IJ\rightarrow R/J$ defined by $\psi_1(a+IJ)=a+I$ and $\psi_2(a+IJ)=a+J$. Because $I$ and $J$ are relatively prime, we know that the function $f\colon R/IJ\rightarrow R/I\oplus R/J$ defined by $f\colon a+IJ\mapsto (a+I,a+J)=(\psi_1(a+IJ),\psi_2(a+IJ))$ is a ring isomorphism. Now, $a+IJ\in L_r(R/IJ)$ if and only if $a+i+IJ\in (R/IJ)^\times$ for all $i\in\{0,1,\ldots,r-1\}$. This occurs if and only if $\psi_1(a+IJ)\in L_r(R/I)$ and $\psi_2(a+IJ)\in L_r(R/J)$, which occurs if and only if $f(a+IJ)\in L_r(R/I)\times L_r(R/J)$. Thus, there is a bijection between $L_r(R/IJ)$ and $L_r(R/I)\times L_r(R/J)$, so $\vert L_r(R/IJ)\vert=\vert L_r(R/I)\vert\cdot \vert L_r(R/J)\vert$. 
 
\end{proof}  
\begin{theorem} \label{Thm2.2} 
Let $P$ be a prime ideal of $R$ such that $R/P$ is finite. Let $r$ and $\alpha$ be positive integers. Then 
\[\mathscr{S}_r(R/P^{\alpha})=
\begin{cases} \vert R/P\vert^{\alpha-1}(\vert R/P\vert-r), & \mbox{if } r\leq char(R/P); \\ \vert R/P\vert^{\alpha-1}(\vert R/P\vert-char(R/P)), & \mbox{if } r>char(R/P). \end{cases}\]
\end{theorem} 
\begin{proof} 
If we define a ring homomorphism $\psi\colon R/P^{\alpha}\rightarrow R/P$ by $\psi(a+P^{\alpha})=a+P$, then we see that $a+P^{\alpha}\in (R/P^{\alpha})^\times$ if and only if $\psi(a+P^{\alpha})\in (R/P)^\times$. More generally, $a+i+P^{\alpha}\in (R/P^{\alpha})^\times$ for all $i\in\{0,1,\ldots,r-1\}$ if and only if $\psi(a+i+P^{\alpha})\in (R/P)^\times$ for all $i\in\{0,1,\ldots,r-1\}$. Therefore, $a+P^{\alpha}\in L_r(R/P^{\alpha})$ if and only if $\psi(a+P^{\alpha})\in L_r(R/P)$. As $\psi$ is clearly surjective and $\displaystyle{\frac{\vert R/P^{\alpha}\vert}{\vert R/P\vert}=\vert R/P\vert^{\alpha-1}}$, we know that $\psi$ is a $k$-to-$1$ mapping, where $k=\vert R/P\vert^{\alpha-1}$. This shows that $\vert L_r(R/P^{\alpha})\vert$ $=\vert R/P\vert^{\alpha-1}\vert L_r(R/P)\vert$. Since $P$ is a prime ideal of the Dedekind domain $R$, $P$ is maximal. Consequently,  $R/P$ is a field. It follows that $(R/P)\backslash (L_r(R/P))=\{P, -1+P,\ldots,-(r-1)+P\}$ if $r\leq char(R/P)$ and $(R/P)\backslash(L_r(R/P))=\{P,-1+P,\ldots,-(char(R/P)-1)+P\}$ if $r>char(R/P)$. Thus, \[\vert L_r(R/P)\vert=
\begin{cases} \vert R/P\vert-r, & \mbox{if } r\leq char(R/P); \\ \vert R/P\vert-char(R/P), & \mbox{if } r>char(R/P). \end{cases}\] 
\end{proof}  
Let $I=P_1^{\alpha_1}P_2^{\alpha_2}\cdots P_k^{\alpha_k}$, where $P_1,P_2,\ldots,P_k$ are distinct prime ideals of $R$ and $\alpha_1,\alpha_2,\ldots,$ $\alpha_k$ are positive integers. Then we may use Theorem \ref{Thm2.1} repeatedly to write
$\mathscr{S}_r(R/I)=$ $\displaystyle{\prod_{i=1}^k\mathscr{S}_r(R/P_i^{\alpha_i})}$ 
for any positive integer $r$. Theorem \ref{Thm2.2} then allows us to evaluate $\mathscr{S}_r(R/P_i^{\alpha_i})$ for each positive integer $i\leq k$. Note that $\mathscr{S}_r(R/(0))$ is defined if $R$ is finite. In this case, $R$ must be a field (any finite integral domain is a field), so an argument similar to that used in the proof of Theorem \ref{Thm2.2} shows that 
\[\mathscr{S}_r(R/(0))=
\vert R\vert-\min(r,char(R)).\] 

We have established one extension of the Schemmel totient functions that is interesting in its own right, but we will see that the following slightly different extension will prove itself much more useful for our purposes later. 
\begin{definition} \label{Def2.2} 
Let $r$ be a nonnegative integer. For each nonzero ideal $I$ of finite index in $R$, we may define the nonnegative integer $\mathcal{S}_r(R/I)$ by the following rules: 
\begin{enumerate}[(a)]
\item $\mathcal{S}_r(R/R)=1$. 
\item If $P$ is a prime ideal of finite index in $R$ and $\alpha$ is a positive integer, then \[\mathcal{S}_r(R/P^{\alpha})=
\begin{cases} \vert R/P\vert^{\alpha-1}(\vert R/P\vert-r), & \mbox{if } r\leq\vert R/P\vert; \\ 0, & \mbox{if } r>\vert R/P\vert. \end{cases}\]
\item If $A$ and $B$ are relatively prime nonzero ideals of finite index in $R$, then $\mathcal{S}_r(R/AB)=\mathcal{S}_r(R/A)\mathcal{S}_r(R/B)$. 
\end{enumerate} 
\end{definition} 
\begin{remark} \label{Rem2.2} 
First, note that we can evaluate $\mathcal{S}_r(R/I)$ for any nonzero ideal $I$ of finite index in $R$ by simply decomposing $I$ into a product of powers of prime ideals and combining the given rules. Then $\mathcal{S}_r(R/I)=0$ if and only if $Q(I)\leq r$. It is easy to see that $\mathcal{S}_0(R/I)=\vert R/I\vert$ and $\mathcal{S}_1(R/I)=\mathscr{S}_1(R/I)=\varphi(R/I)$ for any nonzero ideal $I$ of finite index in $R$. Finally, note that if we set $R=\mathbb{Z}$, then $\mathcal{S}_r(\mathbb{Z}/(n))=\mathscr{S}_r(\mathbb{Z}/(n))=S_r(n)$ for any positive integers $r$ and $n$.  
\end{remark} 
\section{Enumerating Cliques in Generalized Totient Graphs} 
In this section, we will prove our central result, which provides a formula for the number of cliques of order $m$ in any generalized totient graph. From now on, we will always let $I$ denote a nonzero ideal of finite index in the Dedekind domain $R$.

We will need the following lemma. 
 
\begin{lemma} \label{Lem2.3} 
Let $I$ be a nonzero ideal of $R$ such that $R/I$ is finite and nontrivial, and let $r$ be a positive integer. Then $Q(I)\geq r$ if and only if there exist elements $x_1+I,x_2+I,\ldots,x_r+I$ of $R/I$ such that $x_i-x_j+I\in (R/I)^\times$ for all distinct $i,j\in\{1,2,\ldots,r\}$. Furthermore, if $Q(I)\geq r$, then, for any such choice of $x_1+I,x_2+I,\ldots,x_r+I$, there are exactly $\mathcal{S}_r(R/I)$ elements $w+I$ of $R/I$ that satisfy $w-x_i+I\in (R/I)^\times$ for all $i\in\{1,2,\ldots,r\}$. 
\end{lemma} 
\begin{proof} 
First, by way of contradiction, suppose that $Q(I)<r$ and that $x_1+I,x_2+I,\ldots,x_r+I$ are elements of $R/I$ such that $x_i-x_j+I\in (R/I)^\times$ for all distinct $i,j\in\{1,2,\ldots,r\}$. As $Q(I)<r$, there must be some prime ideal divisor $P$ of $I$ such that $\vert R/P\vert<r$. By the pigeonhole principle, we must have $x_i+P=x_j+P$ for some distinct $i,j\in\{1,2,\ldots,r\}$. However, this is a contradiction because it implies that $x_i-x_j\in P$, which then implies that $x_i-x_j+I\not\in (R/I)^\times$.  

Now, suppose $Q(I)\geq r$. Let us write $I=P_1^{\alpha_1}P_2^{\alpha_2}\cdots P_s^{\alpha_s}$, where $P_1,P_2,$ $\ldots,P_s$ are distinct prime ideals of $R$ and $\alpha_1,\alpha_2,\ldots,\alpha_s$ are positive integers. For each positive integer $v\leq s$, we may write $\vert R/P_v\vert=m_v\geq r$ and $R/P_v=\{t_{v,1}+P_v,t_{v,2}+P_v,\ldots,t_{v,m_v}+P_v\}$. For each positive integer $i\leq r$, the Chinese remainder theorem guarantees that it is possible to find some $x_i\in R$ such that $x_i+P_v=t_{v,i}+P_v$ for all positive integers $v\leq s$. Then, for all distinct $i,j\in\{1,2,\ldots,r\}$ and all $v\in\{1,2,\ldots,s\}$, we have $x_i-x_j+P_v=t_{v,i}-t_{v,j}+P_v\neq P_v$, so $x_i-x_j\not\in P_v$. This means that $x_i-x_j+I\in (R/I)^\times$ for all distinct $i,j\in\{1,2,\ldots,r\}$. 

Finally, suppose that $Q(I)\geq r$ and that $x_1+I,x_2+I,\ldots,x_r+I$ satisfy $x_i-x_j+I\in (R/I)^\times$ for all distinct $i,j\in\{1,2,\ldots,r\}$. For each $j\in\{1,2,\ldots,s\}$, let $N_j$ be the set of elements $a+P_j^{\alpha_j}$ of $R/P_j^{\alpha_j}$ such that $a-x_i\not\in P_j$ for all $i\in\{1,2,\ldots,r\}$ (observe that the choices of $a$ and $x_i$ as representatives of their cosets do not affect whether or not $a-x_i\in P_j$ because $P_j\supseteq P_j^{\alpha_j}\supseteq I$). By the Chinese remainder theorem, any element $w+I$ of $R/I$ is uniquely determined by the cosets of the ideals $P_1^{\alpha_1},P_2^{\alpha_2},\ldots,P_s^{\alpha_s}$ that contain the representative $w$. Therefore, the number of ways to choose an element $w+I$ of $R/I$ that satisfies $w-x_i+I\in (R/I)^\times$ for all $i\in\{1,2,\ldots,r\}$ is equal to $\displaystyle{\prod_{j=1}^s\vert N_j\vert}$. 
Fix $j\in\{1,2,\ldots,s\}$, and consider the natural homomorphism $\psi\colon R/P_j^{\alpha_j}\rightarrow R/P_j$ given by $a+P_j^{\alpha_j}\mapsto a+P_j$. We know that $\psi$ is a $k$-to-$1$ mapping, where $k=\vert R/P_j\vert^{\alpha-1}$. An element $A$ of $R/P_j^{\alpha_j}$ is in $N_j$ if and only if $\psi(A)\neq x_i+P_j$ for all $i\in\{1,2,\ldots,r\}$. In other words, $N_j$ is the preimage of the set $B=(R/P_j)\backslash \{x_1+P_j,x_2+P_j,\ldots,x_r+P_j\}$ under $\psi$. Hence, \[\vert N_j\vert=\vert R/P_j\vert^{\alpha-1}\vert B\vert=\vert R/P_j\vert^{\alpha-1}(\vert R/P_j\vert-r)=\mathcal{S}_r(R/P_j^{\alpha_j}).\] 
Using Definition \ref{Def2.2}, we see that the number of elements $w+I$ of $R/I$ that satisfy $w-x_i+I\in (R/I)^\times$ for all $i\in\{1,2,\ldots,r\}$ is \[\prod_{j=1}^s\vert N_j\vert=\prod_{j=1}^s\mathcal{S}_r(R/P_j^{\alpha_j})=\mathcal{S}_r(R/I).\]
 
\end{proof}  

\begin{theorem} \label{Thm3.1} 
For any positive integer $m$, the number of cliques of order $m$ in the graph $G_{R/I}$ is given by the expression 
\[\prod_{k=1}^m\frac{\mathcal{S}_{k-1}(R/I)}{k}.\]
\end{theorem} 
\begin{proof} 
Let $C_m$ be the number of cliques of $G_{R/I}$ of order $m$. The result is trivial if $m=1$ because $\mathcal{S}_0(R/I)=\vert R/I\vert$, which is the number of vertices of $G_{R/I}$. Now, suppose that $m>1$. We will show that $C_m=\dfrac{\mathcal S_{m-1}(R/I)}{m}C_{m-1}$; the theorem will then follow by induction on $m$. If $C_{m-1}=0$, then of course $C_m=\dfrac{\mathcal S_{m-1}(R/I)}{m}C_{m-1}$ because there are no cliques of order $m$. Therefore, let us assume that there is at least one clique of order $m-1$ in $G_{R/I}$, say $D$. Let the vertices in $D$ be $x_1+I,x_2+I,\ldots,x_{m-1}+I$. Then $x_i-x_j+I\in (R/I)^\times$ for all distinct $i,j\in\{1,2,\ldots,m-1\}$ because $D$ is a clique. By Lemma \ref{Lem2.3}, there are exactly $\mathcal{S}_{m-1}(R/I)$ elements $w+I$ of $R/I$ that satisfy $w-x_i+I\in (R/I)^\times$ for all $i\in\{1,2,\ldots,m-1\}$. In other words, there are 
exactly $\mathcal{S}_{m-1}(R/I)$ vertices of $G_{R/I}$ that we may annex to $D$ to make a clique of order $m$. This counts each clique of order $m$ a total of $m$ times because there are $m$ subcliques of order $m-1$ in every clique of order $m$. Hence, $C_m=\dfrac{\mathcal S_{m-1}(R/I)}{m}C_{m-1}$ as desired. 
\end{proof} 
\begin{corollary} \label{Cor3.2} 
Let $m$ and $n$ be positive integers with $n>1$. The number of cliques of order $m$ in the Euler totient Cayley graph $G_{\mathbb{Z}/(n)}$ is \[\prod_{k=1}^m\frac{S_{k-1}(n)}{k}.\] 
\end{corollary} 
Observe that if we set $m=3$ in Corollary \ref{Cor3.2}, we recover the previously--discovered formula $\frac{1}{6}n\phi(n)S_2(n)$ for the number of triangles in $G_{\mathbb{Z}/(n)}$.  
Our proof of this formula seems much more natural and illuminating than those already in existence \cite{Dejter95,Klotz07,Madhavi10}. Before proceeding to uncover some additional properties of generalized totient graphs, we pause to note an interesting divisibility relationship that arises as a corollary of Theorem \ref{Thm3.1}. 
\begin{corollary} \label{Cor3.1} 
For any positive integer $m$, we have 
\[m!\Bigl\lvert\prod_{k=1}^{m}\mathcal{S}_{k-1}(R/I).\]
\end{corollary} 
\section{Other Properties of Generalized Totient Graphs} 
For any graph $G$, let $\chi(G)$ and $\chi'(G)$ denote the chromatic number and the chromatic index of $G$, respectively. Let $g(G)$, $d(G)$, $\Delta(G)$, and $\omega(G)$ denote, respectively, the girth, diameter, maximum degree, and clique number of $G$. A set $D$ of vertices of $G$ is said to dominate $G$ if every vertex in $G$ is either in $D$ or is adjacent to at least one element of $D$. The clique domination number $\gamma_{cl}(G)$ is the smallest positive integer such that there exists a clique of $G$ of order $\gamma_{cl}(G)$ that dominates $G$ (provided some dominating clique exists).

We will continue to let $I$ denote a nonzero ideal of the Dedekind domain $R$ such that $R/I$ is both finite and nontrivial. We begin with a fairly basic lemma concerning unitary Cayley graphs of commutative rings with unity. A symmetric graph is a graph $G$ such that if $A,B,C,D$ are vertices of $G$ with $A$ adjacent to $B$ and $C$ adjacent to $D$, then there exists an automorphism of $G$ that maps $A$ to $C$ and maps $B$ to $D$. 
\begin{lemma} \label{Thm4.1} 
Let $X$ be a commutative ring with unity. The unitary Cayley graph $G_X$ is symmetric and $\varphi(X)$-regular, where $\varphi(X)=\vert X^\times\vert$. 
\end{lemma} 
\begin{proof} 
Choose some $A,B,C,D\in X$ such that $A$ is adjacent to $B$ and $C$ is adjacent to $D$. Define a function $F\colon X\rightarrow X$ by \[F(Z)=C-(A-Z)(C-D)(A-B)^{-1}.\] Observe that $F(A)=C$ and $F(B)=D$. It is straightforward to see that $F$ is a bijection because $(A-B)$ and $(C-D)$ are units in $X$. Now, let $Y$ and $Z$ be any adjacent vertices in $G_X$. It follows from the fact that $Y-Z$, $A-B$, and $C-D$ are all units in $X$ that $F(Y)-F(Z)=(Y-Z)(C-D)(A-B)^{-1}\in X^\times$. Hence, $F(Y)$ and $F(Z)$ are adjacent. Similarly, $F$ maps nonadjacent vertices to nonadjacent vertices. This shows that $F$ is an automorphism, so $G_X$ is symmetric. 
As any symmetric graph is regular and the degree of the vertex $0$ is $\varphi(X)$, we see that $G_X$ is $\varphi(X)$-regular.   
\end{proof} 
One of the first results proved about unitary Cayley graphs states that if $n\in\mathbb Z^+$, then $G_{\mathbb Z/(n)}$ is bipartite if and only if $n$ is even \cite{Dejter95}. The proof is fairly straightforward, and it generalizes immediately to generalized totient graphs. For this reason, we record the following fact and omit the proof. 
\begin{fact} \label{Thm4.2} 
The generalized totient graph $G_{R/I}$ is bipartite if and only if $Q(I)=2$. 
\end{fact} 
Another standard result concerning unitary Cayley graphs states that the clique number and the chromatic number of $G_{\mathbb Z/(n)}$ are both equal to the smallest prime factor of $n$ \cite{Klotz07}. Again, the proof generalized in a straightforward manner, so we omit the proof of the next fact. We will remark, however, that the fact that $\omega(G_{R/I})=Q(I)$ follows immediately from Theorem \ref{Thm3.1} and the fact that $\mathcal{S}_r(R/I)=0$ if and only if $Q(I)\leq r$.
\begin{fact} \label{Thm4.3} 
We have $\omega(G_{R/I})=\chi(G_{R/I})=Q(I)$.   
\end{fact} 
Before attempting to determine the chromatic indices of generalized totient graphs, we need two lemmas that should make the proof of the following theorem relatively painless. 
\begin{lemma} \label{Lem4.1}
Let $m$ be a positive even integer. Let $G$ be a simple $m$-partite graph with partite sets $A_1,A_2,\ldots,A_m$, all of the same cardinality, such that for all distinct $i,j,k\in\{1,2,\ldots,m\}$ and any $v\in A_i$, $v$ is adjacent to exactly as many vertices in $A_j$ as it is to vertices in $A_k$.  
Then $\chi'(G)=\Delta(G)$. 
\end{lemma} 
\begin{proof} 
Because of the trivial inequality $\chi'(G)\geq\Delta(G)$, it suffices to exhibit a proper edge-coloring of $G$ with $\Delta(G)$ colors. Let $V$ be a vertex of $G$ of degree $\Delta(G)$. Without loss of generality, we may assume that $V\in A_m$. Suppose $V$ is adjacent to exactly $t$ vertices in $A_1$. Then, by property $(b)$, $V$ is adjacent to exactly $t$ vertices in $A_i$ for each $i\in\{1,2,\ldots,m-1\}$. As $V$ is not adjacent to any vertices in $A_m$, we see that $\Delta(G)=(m-1)t$. Hence, we may label our $\Delta(G)$ colors $C_{\mu,\lambda}$, where $\mu$ ranges over the set $\{1,2,\ldots,m-1\}$ and $\lambda$ ranges over the set $\{1,2,\ldots,t\}$. Now, let $b_1,b_2,\ldots,b_m$ be the vertices of the complete graph $K_m$. It is well-known that, because $m$ is even, it is possible to properly color the edges of $K_m$ with $m-1$ colors, say $c_1,c_2,\ldots,c_{m-1}$. For any distinct $i,j\in\{1,2,\ldots,m\}$, let $f(i,j)$ be the integer such that $c_{f(i,j)}$ is the color used to color the edge connecting $b_i$ and $b_j$. 

We now describe how to color the edges of $G$. For any distinct $i,j\in\{1,2,\ldots,m\}$, let $H_{i,j}$ be the subgraph of $G$ induced by the vertices in $A_i\cup A_j$. Every such graph $H_{i,j}$ is clearly a bipartite graph with maximum degree at most $t$. Hence, K\"{o}nig's line coloring theorem implies that, for any distinct $i,j\in\{1,2,\ldots,m\}$, it is possible to properly color the edges of $H_{i,j}$ with only the colors in the set $\{C_{f(i,j),\lambda}\colon\lambda\in\{1,2,\ldots,t\}\}$. Doing so for all subgraphs $H_{i,j}$ yields a proper coloring of $G$ with $\Delta(G)$ colors. 
\end{proof} 
\begin{lemma} \label{Lem4.2} 
Let $\delta$ be a positive integer. Any $\delta$-regular simple graph $G$ with an odd number of vertices has chromatic index $\chi'(G)=\delta+1$.  
\end{lemma} 
\begin{proof} 
Let $G$ be a $\delta$-regular simple graph with $m$ vertices, where $m$ is an odd positive integer. If there exists a proper edge-coloring of $G$ with $s$ colors, then no color can be used more than $\displaystyle{\left\lfloor m/2\right\rfloor}$ times. This implies that the number of edges of $G$ cannot exceed $\displaystyle{s\left\lfloor m/2\right\rfloor}$. As the number of edges of $G$ is $m\delta/2$, we have $m\delta/2\leq s\left\lfloor m/2\right\rfloor=s(m-1)/2$. Consequently, $\displaystyle{s\geq \delta\frac{m}{m-1}>\delta}$. By the definition of $\chi'$, there exists a proper edge-coloring of $G$ with $\chi'(G)$ colors, so $\chi'(G)>\delta$. Vizing's theorem states that $\chi'(G)\leq\delta+1$, so we conclude that $\chi'(G)=\delta+1$. 
\end{proof}  
\begin{theorem} \label{Thm4.4} 
If $I$ has a prime ideal divisor whose index in $R$ is a power of $2$, then $\chi'(G_{R/I})=\varphi(R/I)$. Otherwise, $\chi'(G_{R/I})=\varphi(R/I)+1$.
\end{theorem} 
\begin{proof} 
First, suppose $I$ has a prime ideal divisor $P$ such that $\vert R/P\vert=2^k=m$ for some positive integer $k$. Let $R/P=\{a_1+P, a_2+P,\ldots,a_m+P\}$. Let us write $I=P^{\alpha}J$, where $\alpha$ is a positive integer and $J\not\subseteq P$. Define a homomorphism $\psi\colon R/I\rightarrow R/P$ by $\psi\colon a+I\mapsto a+P$, and, for each positive integer $i\leq m$, let $A_i=\psi^{-1}(a_i+P)=\{a+I\in R/I\colon a-a_i\in P\}$. The sets $A_i$, all of which have the same cardinality, partition the vertices of $G_{R/I}$. Furthermore, no two vertices in the same set $A_i$ are adjacent. Hence, $G_{R/I}$ satisfies property $(a)$ of Lemma \ref{Lem4.1}. Now, fix some distinct $i,j\in\{1,2,\ldots,m\}$ and some vertex $v+I\in A_i$. There are exactly $m^{\alpha-1}$ elements $w+P^{\alpha}$ of $R/P^{\alpha}$ such that $w-a_j\in P$, and every one of those elements satisfies $w-v\not\in P$ because $v-a_i\in P$ and $a_i-a_j\not\in P$. Also, if we view the set $v+(R/J)^\times=\{v+s+J\in R/J\colon s+J\in (R/J)^\times\}$ as a coset of the subgroup $(R/J)^\times$ of $R/J$, then we see that there are exactly $\varphi(R/J)$ elements $w+J$ of $R/J$ such that $w-v+J\in (R/J)^\times$. A vertex $w+I$ of $G_{R/I}$ is an element of $A_j$ that is adjacent to $v+I$ if and only if $w-a_j\in P$, $w-v\not\in P$, and $w-v+J\in (R/J)^\times$. Because $R/I\cong R/P^{\alpha}\oplus R/J$, we see that there are exactly $m^{\alpha-1}\varphi(R/J)$ such vertices. This number does not depend on the choice of $j$ (so long as $j\neq i$), so $G_{R/I}$ satisfies property $(b)$ of Lemma \ref{Lem4.1}. Therefore, by Lemmas \ref{Thm4.1} and \ref{Lem4.1}, $\chi'(G_{R/I})=\Delta(G_{R/I})=\varphi(R/I)$.

Now, suppose that $I$ does not have a prime ideal divisor whose index in $R$ is a power of $2$. Let us write $I=Q_1^{\alpha_1}Q_2^{\alpha_2}\cdots Q_t^{\alpha_t}$, where $Q_1,Q_2,\ldots,Q_t$ are prime ideals of $R$ and $\alpha_1,\alpha_2,\ldots,\alpha_t$ are positive integers. For each positive integer $i\leq t$, $R/Q_i$ is a finite field, so $\vert R/Q_i\vert$ must be a power of an odd prime. Hence, 
$\displaystyle{\vert R/I\vert=\prod_{i=1}^t\vert R/Q_i\vert^{\alpha_i}}$
is odd, so $G_{R/I}$ has an odd number of vertices. Using Lemma \ref{Thm4.1} and Lemma \ref{Lem4.2}, we conclude that $\chi'(G_{R/I})=\varphi(R/I)+1$.        
\end{proof}      
\begin{theorem} \label{Thm4.5} 
Let $\lambda(I)$ denote the number of distinct prime ideal divisors of $I$. If $I$ is a prime ideal, then $\gamma_{cl}(G_{R/I})=1$. If $I$ is not a prime ideal and $Q(I)>\lambda(I)$, then $\gamma_{cl}(G_{R/I})=\lambda(I)+1$. If $I$ is not prime and $Q(I)\leq\lambda(I)$, then $\gamma_{cl}(G_{R/I})$ does not exist. Furthermore, if $\gamma_{cl}(G_{R/I})$ exists, then every clique of $G_{R/I}$ of order $\gamma_{cl}(G_{R/I})$ dominates $G_{R/I}$. 
\end{theorem}
\begin{proof} 
Because $R$ is a Dedekind domain, $R/I$ is a field if and only if $I$ is prime. In other words, $G_{R/I}$ is complete if and only if $I$ is prime. Therefore, if $I$ is prime, any single vertex of $G_{R/I}$ forms a dominating clique. Now, suppose $I$ is not prime. Let $P_1,P_2,\ldots,P_{\lambda(I)}$ be the prime ideal divisors of $I$, and assume that $C=\{v_1+I,v_2+I,\ldots,v_t+I\}$ is a clique of $G_{R/I}$ of order $t\leq\lambda(I)$. If $t=1$, then we know we may find some vertex of $G_{R/I}$ other than $v_1+I$ that is not adjacent to $v_1+I$ because $G_{R/I}$ is regular and not complete. Therefore, no clique of order $1$ dominates $G_{R/I}$. Suppose $t>1$. By the Chinese remainder theorem, we know that we may find some $z+I\in R/I$ such that $z+P_i=v_i+P_i$ for all positive integers $i\leq t$. This implies that $z+I$ is not adjacent to any element of $C$. Then, because any vertex in $C$ is adjacent to all other vertices in $C$, $z+I$ cannot be in $C$. Thus, $C$ does not dominate $G_{R/I}$, so we conclude that no clique of order $t\leq\lambda(I)$ can dominate $G_{R/I}$ when $I$ is not prime. 

Now, suppose $I$ is not prime and $Q(I)>\lambda(I)$. Again, let $P_1,P_2,\ldots,P_{\lambda(I)}$ be the prime ideal divisors of $I$. There exists at least one clique of $G_{R/I}$ of order $\lambda(I)+1$ because $\lambda(I)+1\leq Q(I)=\omega(G_{R/I})$ (by Fact \ref{Thm4.3}), so we may let $D$ be an arbitrary clique of $G_{R/I}$ of order $\lambda(I)+1$. We will show that $D$ dominates $G_{R/I}$. Suppose, for the sake of finding a contradiction, that there is some vertex $z+I\in R/I$ that is not adjacent to any of the vertices of $D$. By the pigeonhole principle, there must be some prime ideal divisor $P_i$ of $I$ and some distinct $a+I,b+I\in D$ such that $z-a\in P_i$ and $z-b\in P_i$. Then $a-b\in P_i$, which contradicts the fact that $a-b+I\in (R/I)^\times$ because $D$ is a clique. 
\par 
Finally, suppose $I$ is not prime and $Q(I)\leq\lambda(I)$. Because the clique number of $G_{R/I}$ is $Q(I)$ by Fact \ref{Thm4.3}, we see that there are no cliques of order larger than $\lambda(I)$. Thus, no clique of $G_{R/I}$ can dominate $G_{R/I}$, so $\gamma_{cl}(G_{R/I})$ does not exist. 
\end{proof} 
If $Q(I)\geq 3$, then $g(G_{R/I})=3$ because $\omega(G_{R/I})\geq 3$ by Fact \ref{Thm4.3}. On the other hand, if $Q(I)=2$, then Fact \ref{Thm4.2} implies that $G_{R/I}$ contains no odd cycles. Therefore, if $Q(I)=2$, then $g(G_{R/I})\geq 4$. The following theorem shows that $g(G_{R/I})\leq 4$ for many generalized totient graphs $G_{R/I}$.  
\begin{theorem} \label{Thm4.6} 
Let $P=(p)$ be a prime principal ideal of $R$, and let $J$ be an ideal of $R$ that is not contained in $P$. Let $I=P^{\alpha}J$, where $\alpha$ is a positive integer. If $\alpha>1$ or if $p-1,p+1\not\in J$, then $G_{R/I}$ has a cycle of length $4$.  
\end{theorem} 
\begin{proof} 
Suppose $\alpha>1$ or $p-1,p+1\not\in J$. Let $y$ be an element of $J$ that is not in $P$. Consider the vertices $V_1=I$, $V_2=p^2-y+I$, $V_3=p^2+p+I$, and $V_4=p-y+I$. Suppose $V_1$ and $V_2$ are not adjacent. Then $p^2-y$ is an element of some prime ideal divisor $Q$ of $I$. Since $y\not\in P$, $p^2-y\not]in P$. Hence, $Q$ is a prime ideal divisor of $J$. This implies that $p^2\in Q$ because $y\in J\subseteq Q$. We then have $P^2=(p^2)\subseteq Q$, so $Q$ is a prime ideal divisor of $P^2$. As $P$ is a prime ideal, we must have $P=Q\supseteq J$, which contradicts our hypothesis that $J\not\subseteq P$. Hence, $V_1$ and $V_2$ are adjacent. Similar arguments show that 
$V_2$ is adjacent to $V_3$, $V_3$ is adjacent to $V_4$, and $V_4$ is adjacent to $V_1$. Therefore, if $V_1$, $V_2$, $V_3$, and $V_4$ are all distinct, they form a cycle of length $4$. We know that $V_1\neq V_2$, $V_2\neq V_3$, $V_3\neq V_4$, and $V_4\neq V_1$ because no vertex of $G_{R/I}$ can be adjacent to itself. Hence, it suffices to show that $V_1\neq V_3$ and $V_2\neq V_4$. 

Assume $\alpha>1$. As $1\not\in P$ and $P^{\alpha-1}\subseteq P$, $p-1\not\in P^{\alpha-1}$. This implies that $p^2-p\not\in P^{\alpha}$, so $(p^2-y)-(p-y)=p^2-p\not\in I$. Thus, $V_2\neq V_4$. Similarly, $p^2+p\not\in I$, so $V_1\neq V_3$. 

Suppose, now, that $\alpha=1$. Then $p-1\not\in J$, so $(p^2-y)-(p-y)=p(p-1)\not\in (p)J=I$. This shows that $V_2\neq V_4$. Similarly, $p+1\not\in J$, so $p^2+p=p(p+1)\not\in (p)J=I$. This implies that $V_1\neq V_3$.    
\end{proof} 
\begin{corollary} \label{Cor4.1}
If $n\geq 3$ is an integer, then 
$\displaystyle{g(G_{\mathbb{Z}/(n)})=\begin{cases} 3, & \mbox{if } 2\nmid n; \\ 4, & \mbox{if } 2\vert n, n\neq 6; \\ 6, & \mbox{if } n=6. \end{cases}}$
\end{corollary} 
\begin{proof} 
If $2\nmid n$, then it follows from the paragraph immediately preceding Theorem \ref{Thm4.6} that $g(G_{\mathbb{Z}/(n)})=3$ because $Q((n))\geq 3$. It is easy to see that $g(G_{\mathbb{Z}/(6)})=6$ because $G_{\mathbb{Z}/(6)}$ is a cycle of length $6$. Assume, now, that $2\vert n$ and $n\neq 6$. Because $2\vert n$, $g(G_{\mathbb{Z}/(n)})\neq 3$. Write $n=2^{\alpha}m$ for some positive integers $m$ and $\alpha$ with $m$ odd. Setting $p=2$ and $J=(m)$ in Theorem \ref{Thm4.6} shows that there is a cycle of length $4$ in $G_{\mathbb{Z}/(n)}$, so $g(G_{\mathbb{Z}/(n)})=4$. 
\end{proof}  
\par 
\section{Diameters and Disconnectedness}
Klotz and Sander have determined the diameters of all Euler totient Cayley graphs \cite{Klotz07}, so we will do the same for generalized totient graphs. We wish to acknowledge that the proofs of Lemma \ref{Lem6.1}, Theorem \ref{Thm6.1}, and Theorem \ref{Thm6.3} are inspired by proofs that Klotz and Sander used to establish similar results in the specific case when $R=\mathbb{Z}$. 
\begin{definition} \label{Def6.1} 
For each prime ideal $P$ of $R$ and element $s$ of $R$, define $\varepsilon(P,s)$ by \[\varepsilon(P,s)=\begin{cases} 1, & \mbox{if } s\in P; \\ 2, & \mbox{if } s\not\in P. \end{cases}\]
Let $I=P_1^{\alpha_1}P_2^{\alpha_2}\cdots P_t^{\alpha_t}$ be the (unique) factorization of $I$ into a product of powers of distinct prime ideals. We define $F\colon R\rightarrow\mathbb{Z}$ by \[F(s)=\vert R/I\vert\prod_{i=1}^t\left(1-\frac{\varepsilon(P_i,s)}{\vert R/P_i\vert}\right).\] 
\end{definition} 
The function $F$ clearly depends on the choice of $R$ and the choice of $I$, but we trust that this will not lead to confusion. 
\begin{lemma} \label{Lem6.1} 
For any $a,b\in R$, the vertices $a+I$ and $b+I$ of $G_{R/I}$ have $F(a-b)$ common neighbors. 
\end{lemma} 
\begin{proof} 
As before, we will let $I=P_1^{\alpha_1}P_2^{\alpha_2}\cdots P_t^{\alpha_t}$ be the factorization of $I$ into a product of powers of distinct prime ideals. Fix some $a,b\in R$, and note that a vertex $c+I$ is a common neighbor of $a+I$ and $b+I$ in $G_{R/I}$ if and only if $c-a\not\in P_i$ and $c-b\not\in P_i$ for all $i\in\{1,\ldots,t\}$. For each $i\in\{1,\ldots,t\}$, let $N_i$ be the number of common neighbors of $a+P_i^{\alpha_i}$ and $b+P_i^{\alpha_i}$ in $G_{R/P_i^{\alpha_i}}$. It follows from the Chinese remainder theorem that the number of common neighbors of $a+I$ and $b+I$ in $G_{R/I}$ is $N_1N_2\cdots N_t$. Now, choose some $i\in\{1,\ldots,t\}$ so that we may evaluate $N_i$. We use the natural homomorphism $\psi_i\colon R/P_i^{\alpha_i}\rightarrow R/P_i$ defined by $\psi_i\colon v+P_i^{\alpha_i}\mapsto v+P_i$, which we know is a $k$-to-$1$ mapping with $k=\vert R/P_i\vert^{\alpha_i-1}$. A vertex $c+P_i^{\alpha_i}$ is a common neighbor of $a+P_i^{\alpha_i}$ and $b+P_i^{\alpha_i}$ in $G_{R/P_i^{\alpha_i}}$ if and only if $\psi_i(c+P_i^{\alpha_i})\neq a+P_i$ and $\psi_i(c+P_i^{\alpha_i})\neq b+P_i$. Therefore, if we wish to choose $c+P_i^{\alpha_i}$ to be a common neighbor of  $a+P_i^{\alpha_i}$ and $b+P_i^{\alpha_i}$ in $G_{R/P_i^{\alpha_i}}$, then there are exactly $\varepsilon(P_i,a-b)$ elements of $R/P_i$ that cannot be the image of $c+P_i^{\alpha}$ under $\psi_i$. We see that $N_i=\vert R/P_i\vert^{\alpha_i-1}(\vert R/P_i\vert-\varepsilon(P_i,a-b))=\displaystyle{\vert R/P_i^{\alpha_i}\vert\left(1-\frac{\varepsilon(P_i,a-b)}{\vert R/P_i\vert}\right)}$. Hence, \[N_1N_2\cdots N_t=\prod_{i=1}^t\vert R/P_i^{\alpha_i}\vert\left(1-\frac{\varepsilon(P_i,a-b)}{\vert R/P_i\vert}\right)=F(a-b).\]       
\end{proof} 
\begin{theorem} \label{Thm6.1}
If $I$ is a prime ideal of $R$, then $d(G_{R/I})=1$. If $I$ is not a prime ideal of $R$ and $I$ has no prime ideal divisors of index $2$, then $d(G_{R/I})=2$.
\end{theorem} 
\begin{proof}
If $I$ is a prime ideal of $R$, then $R/I$ is a field. It follows that $G_{R/I}$ is complete, so $d(G_{R/I})=1$. 

Suppose $I$ is not a prime ideal of $R$ and $I$ has no prime ideal divisors of index $2$. Let $P$ be a prime ideal divisor of $I$. We may choose some $p\in P$ with $p\not\in I$. The vertices $I$ and $p+I$ are distinct and nonadjacent. This implies that $d(G_{R/I})\geq 2$. Now, for any $a,b\in R$, it is easy to see that $F(a-b)>0$ because the index of each prime ideal divisor of $I$ is at least $3$. It follows that any vertices $a+I$ and $b+I$ have a common neighbor, so $d(G_{R/I})=2$. 
\end{proof}
 
\begin{theorem} \label{Thm6.2}
Suppose that $I$ has exactly $m$ distinct prime ideal divisors of index $2$ in $R$, where $m\geq 1$. Then $G_{R/I}$ has exactly $2^{m-1}$ disconnected components. All such components are bipartite graphs that are isomorphic to each other and whose diameters are at most $3$. 
\end{theorem} 
\begin{proof} 
For any vector $v\in\mathbb{F}_2^{m}$, let $[v]_i$ denote the $i^{\text{th}}$ coordinate of $v$. Let $P_1,P_2,\ldots,P_m$ be the prime ideal divisors of $I$ of index $2$. We define a function $T\colon R/I\rightarrow\mathbb{F}_2^m$ by specifying that for each $i\in\{1,2,\ldots,m\}$ and $a\in R$,
we have \[[T(a+I)]_i=\begin{cases} 0, & \mbox{if } a\in P_i; \\ 1, & \mbox{if } a\not\in P_i. \end{cases}\] It is easy to verify that $T$ is well-defined. If we write $T^{-1}(v)=\{V\in R/I\colon T(V)=v\}$, then the collection of sets $\displaystyle{\{T^{-1}(v)\}_{v\in\mathbb{F}_2^m}}$ forms a partition of $R/I$ into $2^m$ subsets. Let $\textbf{1}$ denote the vector in $\mathbb{F}_2^m$ whose coordinates are all $1$'s. It is easy to see that $T(V+W)=T(V)+T(W)$ for all $V,W\in R/I$. Furthermore, $(R/I)^\times\subseteq T^{-1}(\textbf{1})$. Therefore, two vertices $V$ and $W$ of $G_{R/I}$ can only be adjacent if $T(V)=\textbf{1}+T(W)$. In other words, for each $v\in\mathbb{F}_2^m$, the vertices in $T^{-1}(v)$ are only adjacent to vertices in $T^{-1}(\textbf{1}+v)$. Let $A$ be a set of $2^{m-1}$ vectors in $\mathbb{F}_2^m$ with the property that $v\in A$ if and only if $\textbf{1}+v\not\in A$, and write $A=\{v_1,v_2\ldots,v_{2^{m-1}}\}$. For each $i\in\{1,2,\ldots,2^{m-1}\}$, write $S_i=T^{-1}(v_i)\cup T^{-1}(\textbf{1}+v_i)$, and let $B_i$ be the subgraph of $G_{R/I}$ induced by $S_i$. Because no vertex in $S_i$ is adjacent to any vertex in $S_j$ when $i\neq j$, we see that $G_{R/I}$ is the union of the $2^{m-1}$ disconnected subgraphs $B_1,B_2,\ldots,B_{2^{m-1}}$. Furthermore, each subgraph $B_i$ is bipartite because we may partition the set of vertices of $B_i$ into the sets $T^{-1}(v_i)$ and $T^{-1}(\textbf{1}+v_i)$. 

Choose some $k\in\{1,2,\ldots,2^{m-1}\}$ and some $a+I,b+I\in T^{-1}(v_k)$. Observe that $T(a-b+I)=T(a+I)-T(b+I)=v_k-v_k=\textbf{0}$. In other words $\varepsilon(P_i,a-b)=1$ for all $i\in\{1,2,\ldots,m\}$. This implies that $F(a-b)>0$, so Lemma \ref{Lem6.1} tells us that $a+I$ and $b+I$ have a common neighbor. The same argument shows that any two vertices in $T^{-1}(\textbf{1}+v_k)$ must have a common neighbor. Therefore, the subgraph $B_k$ is connected and has diameter at most $3$. As $k$ was arbitrary, this shows that each subgraph $B_i$ is connected and has diameter at most $3$. Finally, the subgraphs $B_1,B_2,\ldots,B_{2^{m-1}}$ are isomorphic to each other because $G_{R/I}$ is symmetric by Lemma \ref{Thm4.1}.
\end{proof}

Theorem \ref{Thm6.1} gives the diameter of $G_{R/I}$ when $I$ is not divisible by a prime ideal of index $2$. When $I$ is divisible by a prime ideal of index $2$, Theorem \ref{Thm6.2} tells us that $G_{R/I}$ could be a union of several disconnected components. The following theorem determines the diameter of each component of $G_{R/I}$ when $I$ is divisible by a prime ideal of index $2$. 

\begin{theorem} \label{Thm6.3}
Let $m\in\mathbb Z^+$. Suppose $I=P_1^{\alpha_1}P_2^{\alpha_2}\cdots P_m^{\alpha_m}J$, where $P_1,P_2,\ldots,P_m$ are distinct prime ideals of index $2$ in $R$, $\alpha_1,\alpha_2,\cdots,\alpha_m$ are positive integers, and $J$ is an ideal of $R$ that is not divisible by a prime ideal of index $2$. If $J=R$ and $\alpha_i=1$ for all $i\in\{1,2,\ldots,m\}$, then the diameter of each component of $G_{R/I}$ is $1$. If $J=R$ and $\alpha_i>1$ for some $i\in\{1,2,\ldots,m\}$, then the diameter of each component of $G_{R/I}$ is $2$. If $J\neq R$, then the diameter of each component of $G_{R/I}$ is $3$.
\end{theorem} 
\begin{proof} 
Preserve the notation from the proof of Theorem \ref{Thm6.2}, and let $\lambda=\alpha_1+\alpha_2+\cdots +\alpha_m$. First, suppose that $J=R$. The number of vertices in $G_{R/I}$ is $\vert R/I\vert=\displaystyle{\prod_{i=1}^m\vert R/P_i\vert^{\alpha_i}=2^{\lambda}}$. We know from the proof of Theorem \ref{Thm6.2} that $T^{-1}(v_1)$ and $T^{-1}(\textbf{1}+v_1)$ form two partite sets of the bipartite component $B_1$. If $a+I\in T^{-1}(v_1)$ and $b+I\in T^{-1}(\textbf{1}+v_1)$, then $T(b-a+I)=T(b+I)-T(a+I)=(\textbf{1}+v_1)-v_1=\textbf{1}$. This implies that $b-a\not\in P_i$ for all $i\in\{1,2,\ldots,m\}$, so $a+I$ and $b+I$ are adjacent. It follows that $B_1$ is a complete bipartite graph. Theorem \ref{Thm6.2} tells us that $G_{R/I}$ has $2^{m-1}$ disconnected components that are all isomorphic to each other, so $B_1$ must contain exactly $\displaystyle{\frac{2^{\lambda}}{2^{m-1}}=2^{\lambda-m+1}}$ vertices. If $\alpha_i=1$ for all $i\in\{1,2,\ldots,m\}$, then $B_1$ has only two vertices. In this case, the diameter of $B_1$ is $1$ because $B_1$ is isomorphic to the complete bipartite graph $K_{1,1}$. If $\alpha_i>1$ for some $i\in\{1,\ldots,m\}$, then $\lambda>m$. In this case, $B_1$ is a complete bipartite graph with at least four vertices, so it must have diameter $2$. Because all of the components of $G_{R/I}$ are isomorphic to $B_1$, this completes the proof of the case in which $J=R$. 

Suppose, now, that $J\neq R$. The Chinese remainder theorem guarantees that we may choose some $y\in J$ such that $y\not\in P_i$ for all $i\in\{1,2,\ldots,m\}$. There exists some $\ell\in\{1,2,\ldots,2^{m-1}\}$ such that $S_{\ell}=T^{-1}(\textbf{0})\cup T^{-1}(\textbf{1})$. Because $T(I)=\textbf{0}$ and $T(y+I)=\textbf{1}$, we know that $I$ and $y+I$ are in the same component $B_{\ell}$. The vertices $I$ and $y+I$ are not adjacent because $y\in J\neq R$. Also, Lemma \ref{Lem6.1} implies that $I$ and $y+I$ have no common neighbors because $\varepsilon(P_1,y)=2$. Hence, the diameter of $B_{\ell}$ is at least $3$. By Theorem \ref{Thm6.2}, the diameter of $B_{\ell}$ must be equal to $3$. The proof then follows from the fact that all components of $G_{R/I}$ are isomorphic to $B_{\ell}$.    
\end{proof} 
\section{Strong Colorings of Generalized Totient Graphs} 
Suppose we are given a positive integer $k$ and a graph $G$ with $n$ vertices. Let $\ell$ be the least nonnegative integer such that $k\vert\ell+n$, and let $H$ be the graph that results from adding $\ell$ isolated vertices to $G$. We say that $G$ is strongly $k$-colorable if, for any given partition of the vertices of $H$ into subsets of size $k$, it is possible to properly color the vertices of $H$ so that each color appears exactly once in each subset of the partition. Observe that if $G$ is simple and has some vertex $v$ of degree at least $k$, then we can choose to partition the vertices of $H$ into subsets of size $k$ so that one of the subsets is contained in the neighborhood of $v$. As no vertex in the neighborhood of $v$ can have the same color as $v$ in a proper coloring of $H$, we see that any simple graph with a vertex of degree at least $k$ cannot be strongly $k$-colorable. The strong chromatic number of a graph $G$, denoted $s\chi(G)$, is the smallest positive integer $k$ such that $G$ is strongly $k$-colorable. It follows from the preceding discussion that $s\chi(G)$ must be greater than $\Delta(G)$ if $G$ is simple. 
\par 
We say that an edge-coloring of a graph is strong if any two distinct edges with adjacent endpoints are colored differently. The strong chromatic index of a graph $G$, denoted $s'(G)$, is the smallest positive integer $r$ such that it is possible to strongly edge-color $G$ with $r$ colors. 
\par 
In this section, we briefly study the strong chromatic numbers and strong chromatic indices of some generalized totient graphs.
\begin{theorem} \label{Thm5.1}
Let $P$ be a nonzero prime ideal of $R$ such that $R/P$ is finite, and let $\alpha$ be a positive integer. Then $s\chi(G_{R/P^{\alpha}})=\vert R/P\vert^{\alpha}$. 
\end{theorem} 

\begin{proof} 
For convenience, we write $y=\vert R/P\vert$. Note that $G_{R/P^{\alpha}}$ has $y^{\alpha}$ vertices. We know that $s\chi(G_{R/P^{\alpha}})>\Delta(G_{R/P^{\alpha}})=\varphi(R/P^{\alpha})$, and Theorem \ref{Thm2.2} tells us that $\varphi(R/P^{\alpha})=y^{\alpha}-y^{\alpha-1}$. Let $k=y^{\alpha}-y^{\alpha-1}+m$ for some positive integer $m<y^{\alpha-1}$, and assume by way of contradiction that $G_{R/P^{\alpha}}$ is strongly $k$-colorable. It is easy to see that $\ell=2k-y^{\alpha}$ is the least nonnegative integer such that $\ell+y^{\alpha}$ is divisible by $k$. Let $H$ be the graph that results from adding $\ell$ isolated vertices $v_1,v_2,\ldots,v_{\ell}$ to $G_{R/P^{\alpha}}$. Let $A=\{v+P^{\alpha}\in R/P^{\alpha}\colon v\in P\}$, and let $B=\{v_1,v_2,\ldots,v_r\}$, where $r=k-y^{\alpha-1}=y^{\alpha}-2y^{\alpha-1}+m$. Then $\vert A\cup B\vert=k$ because $\vert A\vert=y^{\alpha-1}$. If we let $C$ be the set of $k$ vertices of $H$ that are not in $A\cup B$, then we see that the two sets $A\cup B$ and $C$ form a partition of the set of vertices of $H$ into subsets of size $k$. Because $G_{R/P^{\alpha}}$ is strongly $k$-colorable, it is possible to properly color the vertices of $H$ so that each color appears exactly once in $A\cup B$ and exactly once in $C$. 

Let $s+P^{\alpha}\in A$ be colored with the color $c$. We know that $c$ must be used to color exactly one vertex $V\in C$. Suppose $V=u+P^{\alpha}\in G_{R/P^{\alpha}}$. Then $u\not\in P$ because $V\not\in A$. However, $s\in P$, so $s-u\not\in P$. This shows that $s-u+P^{\alpha}\in U(R/P^{\alpha})$, so $V$ is adjacent to $s+P^{\alpha}$. This is a contradiction because we assumed the coloring to be proper. Thus, each of the $y^{\alpha-1}$ colors used to color the elements of $A$ must be used to color one of the elements of the set $D=\{v_{r+1},v_{r+2},\ldots,v_{\ell}\}$. However, this is impossible because $\vert D\vert=\ell-r=2k-y^{\alpha}-(y^{\alpha}-2y^{\alpha-1}+m)=2(y^{\alpha}-y^{\alpha-1}+m)-y^{\alpha}-(y^{\alpha}-2y^{\alpha-1}+m)=m<y^{\alpha-1}$. Hence, we must have $s\chi(G_{R/P^{\alpha}})\geq y^{\alpha}$. Clearly, $G_{R/P^{\alpha}}$ is strongly $y^{\alpha}$-colorable, so $s\chi(G_{R/P^{\alpha}})=y^{\alpha}=\vert R/P\vert^{\alpha}$. 
\end{proof} 
\begin{theorem} \label{Thm5.2}
Let $P$, $Q$, and $M$ be nonzero ideals of $R$ such that $P$ and $Q$ are prime, $M\not\subseteq Q$, $R/P$ and $R/M$ are finite and nontrivial, and $\vert R/Q\vert=2$. Let $\alpha$ be a positive integer. We have \[s'(G_{R/P^{\alpha}})=\frac{1}{2}\vert R/P\vert^{\alpha}\varphi(R/P^{\alpha})=\frac{1}{2}\vert R/P\vert^{2\alpha-1}(\vert R/P\vert-1)\] and \[s'(G_{R/QM})\leq\frac{1}{2}\vert R/M\vert\varphi(R/M).\]     
\end{theorem} 
\begin{proof} 
By Theorem \ref{Thm4.5}, any two adjacent vertices of $G_{R/P^{\alpha}}$ dominate $G_{R/P^{\alpha}}$. Therefore, in any strong edge-coloring of $G_{R/P^{\alpha}}$, no two distinct edges can have the same color. This means that any strong edge-coloring of $G_{R/P^{\alpha}}$ must use exactly $\frac{1}{2}\vert R/P\vert^{\alpha}\varphi(R/P^{\alpha})$ colors because $G_{R/P^{\alpha}}$ has $\frac{1}{2}\vert R/P\vert^{\alpha}\varphi(R/P^{\alpha})$ edges. 

Note that $2\in Q$ since $R/Q\cong\mathbb F_2$. Consider the graph $G_{R/QM}$, which has an edge set of size \[\frac{1}{2}\vert R/QM\vert\varphi(R/QM)=\frac{1}{2}\vert R/Q\vert\cdot\vert R/M\vert\varphi(R/Q)\varphi(R/M)=
\vert R/M\vert\varphi(R/M).\] Choose some $\mu\in M$ with $\mu\not\in Q$. For any vertex $v+QM$ of $G_{R/QM}$, let $f(v+QM)=v+\mu+QM$. For any edge $\{Y,Z\}$ of $G_{R/QM}$, let $h(\{Y,Z\})=\{f(Y),f(Z)\}$. Notice that, for any vertex $Y$ of $G_{R/QM}$, we have $f(f(Y))=Y$ because $2\mu\in QM$. Furthermore, it is easy to see that $f(Y)$ is not equal to $Y$ or adjacent to $Y$. This shows that $h(h(E))=E\neq h(E)$ for all edges $E$ of $G_{R/QM}$. Let us color the edges of $G_{R/QM}$ so that two distinct edges $e_1$ and $e_2$ have the same color if and only if $e_2=h(e_1)$. Because each color is used to color two of the edges of $G_{R/QM}$, this coloring uses $\frac{1}{2}\vert R/M\vert\varphi(R/M)$ colors. 

To show that this coloring is strong, suppose $a+QM$ and $b+QM$ are adjacent vertices of $G_{R/QM}$. We stated already that $a+QM$ is not adjacent to $f(a+QM)$. Because $\mu\not\in Q$ and $a-b\not\in Q$, we must have $\mu+Q=a-b+Q=1+Q$. Therefore, $a-[b+\mu]\in Q$, which means that $a+QM-f(b+QM)\not\in (R/QM)^\times$. Hence, $a+QM$ is not adjacent to $f(b+QM)$. By the same token, $b+QM$ is not adjacent to $f(a+QM)$ or $f(b+QM)$. This shows that, for any edge $E$, the only other edge with the same color as $E$ cannot have an endpoint adjacent to an endpoint of $E$. Thus, the coloring is strong.
\end{proof} 
\section{Erratum} 
We use this section to briefly expose two mistakes that have appeared in the literature. First, Manjuri and Maheswari have made the claim that if $n$ is a composite odd integer, then $\gamma_{cl}(G_{\mathbb{Z}/(n)})$ exists if and only if $n$ has $2$ or fewer distinct prime factors \cite{Maheswari13a}. However, setting $R=\mathbb{Z}$ and $I=(n)$ in Theorem \ref{Thm4.5} shows that their claim is false.

We now provide a counterexample to Theorem 4.3 of a different paper of Manjuri and Maheswari \cite{Maheswari13b}. We hope to encourage the reader to discover a correct solution to this interesting problem and perhaps even generalize the solution in order to solve an analogous problem concerning generalized totient graphs. We have been able to obtain partial results in this direction, but have been unable to solve the problem in general. 

The domination number $\gamma(G)$ of a graph $G$ is the size of the smallest dominating set of $G$. The paper states that if a composite integer $n$ is not twice a prime, then $\gamma(G_{\mathbb{Z}/(n)})$ is the smallest positive integer $\ell$ such that every set of $\ell$ consecutive integers contains an element that is not relatively prime to $n$ (this number $\ell$ is often denoted $g(n)$, and $g$ is called Jacobsthal's function). One may verify that the set $\{\overline 0,\overline 2,\overline{5},\overline{27}\}$ dominates $G_{\mathbb{Z}/(30)}$ (we let $\overline k=k+(30)$). Therefore, $\gamma(G_{\mathbb Z/(30)})\leq 4$ (in fact, it can be shown that $\gamma(G_{\mathbb Z/(pqr)})=4$ for any distinct primes $p,q,r$). However, $g(30)=6$, so the claim is false. 
\section{Concluding Remarks}

We wish to acknowledge the potential to generalize and strengthen the preceding results. For example, one might wish to study the infinite graphs that arise from eliminating the restriction that $R/I$ be finite. Alternatively, one might attempt to gather information about analogues of gcd-graphs \cite{Klotz07} in Dedekind domains. There is certainly room for strengthening Theorems \ref{Thm5.1} and \ref{Thm5.2}, which only apply to generalized totient graphs with specific properties. Theorem \ref{Thm4.6} leads us to inquire about which generalized totient graphs have girths not equal to $3$ or $4$. Finally, we note that we have obviously not exhausted all of the graph parameters of generalized totient graphs that one might wish to study. 

\section{Acknowledgments}
\noindent Dedicated to my wonderful and inspiring parents, Marc and Susan Defant. 

\noindent This work was funded in part by NSF grant no. 1262930.


\begin{thebibliography}{9} 

\bibitem{Dejter95}
Dejter, I. J., and Giudici, R. E. On unitary Cayley graphs. J. Combin. Math.
Combin. Comput. 18 (1995), 121–124.

\bibitem{Klotz07}
Klotz, Walter; Sander, Torsten. Some properties of unitary Cayley graphs. \emph{Electron. J. Combin.} (2007). 

\bibitem{Madhavi10}
Madhavi, Levaku; Maheswari, Bommireddy. Enumeration of Hamilton cycles and triangles in Euler totient Cayley graphs. Graph Theory Notes N. Y. 59 (2010), 28--31. 

\bibitem{Maheswari13a}
Maheswari, B., Manjuri, M. Clique dominating sets of Euler totient Cayley graphs. IOSR Journal of Mathematics (2013), 46--49. 

\bibitem{Maheswari13b} 
Maheswari, B., Manjuri, M. Strong dominating sets 
of some arithmetic graphs, 
International Journal of Computer Applications (IJCA), 
Volume 83, No3 (2013), 36--40.

\bibitem{Schemmel69}
V. Schemmel, \"{U}ber relative Primzahlen, \emph{Journal f\"{u}r die reine und angewandte Mathematik}, 70 (1869), 191--192. 

\end{thebibliography}
\end{document}